\documentclass[11pt]{article}
\usepackage{amsmath, amsthm, amsfonts, amssymb}

\newtheorem{theorem}{Theorem}[section]
\newtheorem{corollary}[theorem]{Corollary}
\newtheorem{lemma}[theorem]{Lemma}
\newtheorem{proposition}[theorem]{Proposition}

\newtheorem{remark}[theorem]{Remark}
\newtheorem{definition}[theorem]{Definition}

\numberwithin{equation}{section}

\title{\bf Differential Harnack inequalities and Perelman type entropy formulae for subelliptic operators %\thanks{Partially supported by NSF of China No. 11201040.}
}
\author{Bin Qian\thanks{Department of Mathematics, Changshu Institute of Technology,
Changshu, Jiangsu 215500,  China.
 E-mail: binqiancs@yahoo.com}
 }\date{}%October. 2011

\begin{document}
\maketitle

\begin{abstract}
In this paper, under the generalized curvature-dimension inequality  recently introduced by F. Baudoin and N. Garofalo, we obtain differential Harnack inequalities (Theorem \ref{thm-Schd}) for the positive solutions to the Sch\"odinger equation associated to subelliptic operator with potential. As  applications of the differential Harnack inequality, we derive the corresponding parabolic Harnack inequality (Theorem \ref{harnack1}).  Also we define the Perelman type entropy associated to subelliptic operators and derive its monotonicity (Theorem \ref{thm-perel}).
\end{abstract}


{\bf AMS Classification Subjects 2010:} 53C21 58J35

%\section*{Questions}

%Local version??
%1) 例如对于在Heisenberg群上的某个domain如半平面上的局部LiYau不等式是否有什么帮助,能否利用该局部
%微分Harnack不等式得到体积性条件的Liouville性质。

%2) 结合Zhu Xiaobao， Li Ma 和Lin Zhao PAMS的文章，得到p-Laplacian（Porous medium 方程 ）方程的Hamilton梯度估计，进而得到类似于Zhu Xiaobao的Liouville定理

%3) 能否利用这种技术避免BG文中（Version 4）的完备性要求的条件，同时利用PDE中的极大值原理避免BG文中极大值原理的可积性条件。

%4) 将 Junfang LI(Adv. Math.)的文章考虑进去，可否在Perelman熵的表达式中出现类似的项?

%对特征值估计有何用处？已解决，这里我要求\lambda\ge0,而Lu=\lambda' u
%(\lambda'\le 0 since L is negative operator.)

\section{Introduction}\label{sec-intr}

 Let $M$ be a $C^{\infty}$ connected finite dimensional compact
manifold with a smooth measure $\mu$ and a second-order diffusion operator $L$ on $M$, locally subelliptic, satisfying $L1=0$,
$$
\int_MfLgd\mu=\int_M gLfd\mu, \ \ \int_M fLfd\mu\le0
$$
for every $f,g\in C^{\infty}(M)$. In the neighborhood of every point $x\in M$, It can be written as
$$
L=-\sum_{i=1}^mX_i^*X_i,
$$
where $X_1,X_2,\cdots, X_m$ are Lipschitz continuous, see \cite{Jerison,BG}. For any function $f,g\in
C^{\infty}(M)$, the carr\'e du champ  associated to $L$ is
\begin{equation}\label{Gamma1}
\Gamma(f,g)=\frac12(L(fg)-fL g-gL f)=\sum_{i=1}^mX_ifX_ig,
\end{equation}
and $\Gamma(f)=\Gamma(f,f)$.

%Denote  $d(x,y)$ the canonical distance
%(Carnot-Carath\'eodory distance) associated with $L$:
%$$
%d(x,y)=\sup\{|f(x)-f(y)|:f\in C^{\infty}(M), \Gamma(f,f)\le1\}, \
%x,y\in M
%$$
In addition, as in \cite{BG} we assume that $M$ is endowed with
another smooth symmetric bilinear differential form, indicated with
$\Gamma^Z$, satisfying for $f,g\in C^{\infty}(M)$,
$$
\Gamma^Z(fg,h)=f\Gamma^Z(g,h)+g\Gamma^Z(f,h),
$$
and $\Gamma^Z(f)=\Gamma^Z(f,f)$.  We make the
following assumption (see \cite{BG} for detailed explanation) holds
throughout the paper:
%\begin{description}
%\item{(H.1)} There exists an increasing sequence $h_k\in C_{0}^{\infty}(M)$ such that $h_k\uparrow 1$ on $M$, and
%    $$
 %%   \|\Gamma(h_k)\|_{\infty}+ \|\Gamma^Z(h_k)\|_{\infty}\to 0, \mbox{ as } k\to\infty.
 %   $$
%\item{(H.2)}

{\bf Assumption:}  For any $f\in C^{\infty}(M)$, one has
$$
\Gamma(f,\Gamma^Z(f))=\Gamma^Z(f,\Gamma(f)).
$$
%\end{description}

Denote
$$
\Gamma_2(f,g)=\frac12\left(L\Gamma(f,g)-\Gamma(f,Lg)-\Gamma(g,Lf)\right),
$$
$$
\Gamma_2^Z(f,g)=\frac12\left(L\Gamma^Z(f,g)-\Gamma^Z(f,Lg)-\Gamma^Z(g,Lf)\right),
$$
and $\Gamma_2(f)=\Gamma_2(f,f), \Gamma_2^Z(f)=\Gamma_2^Z(f,f)$.

Now we are ready to recall
\begin{definition}\label{def1}(due to \cite{BG}) We  say that $L$ satisfies the
generalized curvature-dimension inequality $CD(\rho_1,\rho_2,k,d)$ if there exists constants $
\rho_1\in \mathbb{R}$, $\rho_2> 0, k\ge0$ and $d\in[2,\infty]$ such
that  the inequality
$$
\Gamma_2(f)+\nu\Gamma_2^Z(f)\ge
\frac{1}{d}(Lf)^2+\left(\rho_1-\frac{k}{\nu}\right)\Gamma(f)+\rho_2\Gamma^Z(f),
$$
holds for all $f\in C^\infty(M)$ and every $\nu>0$.
 \end{definition}

 The above definition generalizes to
 the curvature dimension inequality $CD(K,m)$ introduced by D. Bakry and M. \'Emery, see for example \cite{Ba97}. Indeed,
 we only  need  to take $\rho_1=K$, $\rho_2=k=0$, $d=m$.
  Besides Laplace-Beltrami operators on compact
Riemannian manifolds with Ricci curvature bounded below and sub-Laplacian $\Delta_b$ in a closed pseudohermitian 3-manifold $(M^3,J,\theta)$, there
exists a wide class of examples satisfying the generalized
curvature-dimension condition, e.g. see \cite{BG}.

Denote the heat semigroup $P_t=e^{tL}$, due to the hypo-ellipticity of $L$,
one has the function $(t,x)\to P_tf(x)$ is smooth on $M\times (0,\infty)$
and $P_tf(x)=\int_Mp(x,y,t)f(y)d\mu(y)$, where $p(x,y,t)=p(y,x,t)>0$ is the so-called heat kernel associated to $P_t$.

 Suppose the subelliptic operator $L$ satisfies the generalized curvature-dimension inequality $CD(\rho_1,\rho_2,k,d)$, Baudoin and  his collaborators  (\cite{BG,BB2,BBG,BBGM,BK}) have obtained that 1). the heat semigroup $P_t$ is stochastically complete, i.e. $P_t1=1$; 2). Li-Yau type inequalities: $\Gamma(\log P_tf)+\frac{2}{3}\rho_2t\Gamma^Z(\log P_tf)\le \left(1+\frac{3k}{2\rho_2}-\frac{2\rho_1}{3}t\right)\frac{LP_tf}{P_tf}+\frac{d\rho_1^2}{6}t
 -\frac{d\rho_1}{2}\left(1+\frac{3k}{d\rho_2}\right)+\frac{d}{2t}\left(1+\frac{3k}{d\rho_2}\right)^2$; 3). Scale-invariant parabolic Harnack inequality; 4). Off-diagonal Gaussian upper bounds for the heat kernel $p(x,y,t)$; 5). Liouville type property; 6). Bonnet-Myers type theorem; 7). Volume comparison property, volume doubling property and distance comparison theorem etc.; 8). Functional inequalities such as Poincar\'e inequality, Log-Sobolev inequality; $\cdots\cdots$, which generalizes the works of S. T. Yau,  D. Bakry, M. Ledoux etc., see \cite{LY, Ba97, Ledoux} and references therein. See also \cite{CY,Qian2,Qian3,Qian4} for recent works on  gradient estimate and curvature property for subelliptic operators.

In the fundamental work \cite{LY}, P. Li and S. T. Yau established the wellknown gradient estimates (so-called Li-Yau inequalities) for the positive solutions to the heat equation $\partial_tu=\Delta u$ on the complete Riemannian manifolds, since it becomes a powerful tool in differential geometry, PDE, etc. It also plays an important role in the Perelman's solution to the Poicar\'e conjecture. As mentioned in the above section, F. Baudoin and N. Garofalo \cite{BG} have obtained the Li-Yau type inequality, but the method (semigroup tools) exploited in \cite{BG} (see also \cite{BBBQ} for three dimensional subelliptic models) is not adaptable for the diffusion operators with potential $L^V$ (Sch\"odinger operator).  since $L^V$ does not commute with the Feynman-Kac semigroup $P_t^V$ generated by $L^V$. While the Li-Yau method (maximum principle) exploited in \cite{LY} works for the Sch\"odinger case, hence it would be interesting to find certain method adaptable for the Sch\"odinger equation. This is the start point of this paper.

 This paper is devoted to study differetnial Harnack inequalities (Li-Yau type gradient estimates) for the positive solutions to the heat equation (Sch\"odinger equation) associated to subelliptic operators with potential, see section \ref{schd2} for the statement of Li-Yau type inequalities and section \ref{proof} for the proof. Section \ref{sec-harnack} is devoted to the Harnack inequalities for positive solutions to the associated Sch\"odinger equation. By applying the differential Harnack inequalities obtained in section \ref{schd2}, we study the Perelman type entropy and its monotonicity in section \ref{entropy}.

\section{Differential Harnack inequalities}\label{schd2}
 In this section, we shall study  the  Schr\"odinger equation $\partial_t u=(L-V)u$, where $L$ is the subelliptic diffusion operator as in Section
\ref{sec-intr} and the potential $V=V(x,t) \ge0$ defined on
$M\times[0,\infty)$ is $C^{\infty}$. Denote $L^V=L-V$.

Throughout this section, we assume that $L$ (rather than $L^V$) satisfies the generalized curvature-dimension inequality $CD(\rho_1,\rho_2,k,d)$ and there exists some positive constants
 $\gamma_1,\gamma_2,\theta$ such that
\begin{equation}\label{schd1}
\Gamma(V)\le \gamma_1^2, \Gamma^Z(V)\le \gamma_2^2,  LV\le \theta.
\end{equation}

For a given $C^1$
function $a(t):[0,\infty)\to [0,\infty)$ such that
 $a(0)=0 $, $\lim\limits_{t\to0}\frac{a(t)}{a'(t)}=0$, $\frac{a'}{a}>0$, $\frac{\int_0^ta(s)ds}{a(t)}>0$ and for any $T>0$, $\frac{a'^2}{a}$, $ \frac{a^3(t)}{(\int_0^ta(s)ds)^2}$, $\frac{1}{a'(t)}\left(\int_0^t
\frac{a^2(s)}{\int_0^sa(u)du}ds\right)^2,
\frac{1}{a'(t)}\left(\int_0^ta(s)ds\right)^2 $  are continuous and  integrable on the domain $(0,T]$.  We have the following Li-Yau type inequality for the
Schr\"odinger equation:

\begin{theorem}\label{thm-Schd}Let $u$ be a positive solution to the
Schr\"odinger equation \begin{equation}\label{schd0}\partial_tu=L^V u,\end{equation} such that
$u_0:=u(\cdot,0)$ is $C^{\infty}(M)$ and $u_0\ge0$, we have the following differential
Harnack inequality: for any $\varepsilon_1,\varepsilon_2\in(0,1)$, we have
\begin{equation}\label{LY1} \Gamma(\log u)
+b(t)\Gamma^Z(\log u)
-\alpha(t)\left((\log
u)_t+V\right)\le \varphi(t),
\end{equation}
where
$$
b(t)=\frac{2(1-\epsilon_2)\rho_2\int_0^ta(s)ds}{a(t)},\ \  \eta(t)=\frac{d}{4}\left((1+\varepsilon_1)\frac{a'}{a}+\frac{ka(t)}{\rho_2\int_0^ta(s)ds}-2\rho_1\right), $$
and $$ \alpha(t)=\frac{4}{d\cdot a(t)}\int_0^ta(s)\eta(s)ds,\ \varphi(t)=
\frac{1}{a(t)}\int_0^ta(s)\left(\frac{\gamma_1^2a(s)|\alpha-1|^2}{\epsilon_1
a'(s)}+\frac{b^2\gamma_2^2}{2\epsilon_2\rho_2}+\theta\alpha(s)+\frac{2\eta^2(s)}{d}\right)ds.\nonumber
$$
% In this case, if we choose $a(t)=t^2, the corresponding result \eqref{LY1}
%reduces to Theorem 5.1 in \cite{BG}

\end{theorem}
\begin{remark}\label{rem-LY}
(i) For the elliptic operators (possibly with potential) satisfying the curvature dimension inequality
$CD(\rho,m)$ introduced in \cite{Ba97}, the corresponding result has been obtained in \cite{Qian1}.
Moreover, the local differential Harnack inequalities have been derived in \cite{Qian1}.

(ii) For diffusion operator $L^V=L$ (i.e. $V\equiv0$, see Proposition \ref{thm-LY} below), a similar  result has been obtained by
Baudoin and Garofalo \cite{BG} by different method, see Corollary
4.3 in \cite{BG}. The method in this paper is motivated by the one in
\cite{LX11,Qian1}, it works not only for the diffusion operators  but also
works for the Schr\"odinger operators (i.e. diffusion operators with
potential), see Section \ref{schd2}. While  the one  in \cite{BG} does not work for the Schr\"odinger  operator $L^V$. But due to the validity of the parabolic comparison theorem (or maximum principle), we assume $M$ is compact here.

(iii) In the case of sub-Laplacian $\Delta_b$ in a closed pseudohermitian 3-manifold $(M^3,J,\theta)$, in fact, $\Delta_b$ satisfies $CD(k,\frac12,1,2)$ where $k$ is the lower bound of Tanaka-Webster curvature. The above result generalizes  Theorem 1.1 in \cite{CKL} by Chang et al.
\end{remark}

A possible function $a(t)$ is $a(t)=t^{\gamma}$ with $\gamma>1$. In this case, Theorem \ref{thm-Schd} reduces to

\begin{corollary}\label{cor0}
Under the assumption of Theorem \ref{thm-Schd}, there exists positive constants $C_i\ (i=1,2,3,4)$ depends on $\rho_1,\rho_2,k,d,\gamma_1,\gamma_2,\gamma,\theta$, we have for all $t>0$,
\begin{align}
 \Gamma(\log u)
+\frac{\rho_2}{1+\gamma}t\Gamma^Z(\log u)
\le& \left(\frac32+\frac{(1+\gamma)k}{\gamma\rho_2}-\frac{2\rho_1}{1+\gamma}\right)\left((\log
u)_t+V\right)+\frac{d\left(1.5\gamma\rho_2+(1+\gamma)k\right)^2}{8(\gamma-1)\rho_2^2t}\nonumber\\
&\ -\frac{d\rho_1(1.5\gamma\rho_2+(1+\gamma)k)}{2\gamma\rho_2}+C_1t+C_2t^2+C_3t^3+C_4t^4.\label{cor0-dh}
\end{align}

\end{corollary}

Another possible function $a(t)$ is $a(t)=e^{\gamma \rho_1t}\left(e^{\gamma \rho_1t}-1\right)^2$ with $\gamma\cdot sgn(\rho_1)>0$, taking $\varepsilon_1=\varepsilon_2=\frac12$, Theorem \ref{thm-Schd} reduces to
\begin{corollary}\label{cor1}
Under the assumption of Theorem \ref{thm-Schd}, there exist positive constants $C_1, C_2$ depends on $\rho_1,\rho_2,k,d,\gamma_1,\gamma_2,\gamma,\theta$, we have for all $t>0$,
\begin{align}
 \Gamma(\log u)&
+\frac{\rho_2}{3\gamma \rho_1}\left(e^{\gamma\rho_1t}-1\right)e^{-\gamma\rho_1t}\Gamma^Z(\log u)
\le \alpha(t)\left((\log
u)_t+V\right)+C_1+\frac{C_2}{e^{\gamma\rho_1t}-1},\label{cor1-dh}
\end{align}
where $$
\alpha(t)=\frac{3}{2}+\frac{k}{\rho_2}-\frac{2}{3\gamma}+\left(\frac{2}{3\gamma}+\frac{k}{2\rho_2}\right)e^{-\gamma\rho_1t},\ %\
%\varphi(t)=C_1+\frac{C_2}{\left|e^{\gamma\rho_1t}-1\right|}+\frac{C_3}{\left(e^{\gamma\rho_1t}-1\right)^2}.
$$
\end{corollary}

Comparing with the Corollary \ref{cor0}, \eqref{cor1-dh} works for both small time and large time $t$, while \eqref{cor0-dh} works for the  finite time $t$.

As an application of the above Corollary, we have
\begin{corollary}\label{cor11} Assume $L$ satisfies curvature inequality $CD(\rho_1,\rho_2,k,d)$ with $\rho_1\neq0$. Let $u$ be a positive solution to the equation $L^Vu=0$ with $V$ defined on $M$, there exists a positive constant $C$ such that
\begin{align}
 \Gamma(\log u)&
+\frac{\rho_2}{4 |\rho_1|}\Gamma^Z(\log u)
\le \left(2+\frac{k}{\rho_2}\right)V+C.\label{cor12-dh}
\end{align}

\end{corollary}

In particular, for a given $C^1$ function
 $a(t):[0,\infty)\to [0,\infty)$ such that
 $a(0)=0 $, $\lim\limits_{t\to0}\frac{a(t)}{a'(t)}=0$, $\frac{a'}{a}>0$, $\frac{\int_0^ta(s)ds}{a(t)}>0$  and for any $T>0$, $\frac{a'^2}{a}$, $ \frac{a^3(t)}{(\int_0^ta(s)ds)^2}$
 are continuous and  integrable on the domain $(0,T]$, we have  Li-Yau type differential Harnack inequalities for subelliptic diffusion operator $L$ (i.e. $V\equiv0$).

\begin{proposition}\label{thm-LY} Let $u$ be a positive solution
to the heat equation $\partial_tu=Lu$ such that $u_0:=u(\cdot,0)\in
C^{\infty}(M)$ with $u_0\ge0$, we have the following differential Harnack inequality:
\begin{equation}\label{LY1} \Gamma(\log u)
+\frac{2\rho_2\int_0^ta(s)ds }{a(t)}\Gamma^Z(\log u) -\alpha(t)(\log u)_t\le
\varphi(t),
\end{equation}
where
$$
  \eta(t)=\frac{d}{4}\left(\frac{a'}{a}+\frac{ka(t)}{\rho_2\int_0^ta(s)ds}-2\rho_1\right), \   \alpha(t)=\frac{4
\int_0^ta(s)\eta(s)ds}{d\cdot a(t)},\ \varphi(t)=\frac{2\int_0^ta(s)\eta^2(s)\,ds}{d\cdot a(t)}.
$$

\end{proposition}

Taking $a(t)=t^{\gamma}$ with $\gamma>1$, then
$\eta(t)=\frac{d}{4}\left(\left(\gamma\rho_2+(\gamma+1)k\right)\frac{1}{\rho_2 t}-2\rho_1\right),$
 Proposition \ref{thm-LY} reduces to

\begin{corollary}\label{cor10}
Let $u$ be a positive solution to the heat
equation $\partial_tu=Lu$ such that $u_0=u(\cdot,0)\in C^{\infty}(M)$ with $u_0\ge0$,
we have for $\gamma>1$,
\begin{align}\label{gradient1}
\Gamma(\log u)+\frac{2\rho_2}{1+\gamma}t\Gamma^Z(\log u)\le&
\left(1+\frac{(1+\gamma)k}{\gamma\rho_2}-\frac{2\rho_1}{1+\gamma}t\right)(\log
u)_t+\frac{d\rho_1^2t}{2(1+\gamma)}-\frac{d\rho_1}{2}
\left(1+\frac{(1+\gamma)k}{\gamma\rho_2}\right)\nonumber\\
&+\frac{d\gamma^2}{8(\gamma-1)t}\left(1+\frac{(1+\gamma)k}{\gamma\rho_2}\right)^2.
\end{align}

If $\rho_1\ge0$, it is easy to see \eqref{gradient1} reduces to
 \begin{align}\label{gradient1-10}
\Gamma(\log u)+\frac{2\rho_2}{1+\gamma}t\Gamma^Z(\log u)\le&
\left(1+\frac{(1+\gamma)k}{\gamma\rho_2}\right)(\log
u)_t+\frac{d\gamma^2}{8(\gamma-1)t}\left(1+\frac{(1+\gamma)k}{\gamma\rho_2}\right)^2.
\end{align}
\end{corollary}

\begin{remark}
Choose $\gamma=2$, Corollary \ref{cor1} reduces to
\begin{equation*}
\Gamma(\log u)+\frac{2\rho_2}{3}t\Gamma^Z(\log u)\le \left(1+\frac{ 3k}{2\rho_2}-\frac{2\rho_1t}{3}\right)(\log u)_t+\frac{d\rho_1^2}{6}t-\frac{d\rho_1}{2}\left(1+\frac{ 3k}{2\rho_2}\right)
+\frac{d\left(1+\frac{3k}{2\rho_2}\right)^2}{2t},
\end{equation*}
which is nothing less than theorem 5.1 in \cite{BG}.
In particular, if $\rho_1\ge 0$, we have
\begin{equation}\label{gradient1-100}
\Gamma(\log u)+\frac{2\rho_2}{3}t\Gamma^Z(\log u)\le \left(1+\frac{ 3k}{2\rho_2}\right)(\log u)_t+\frac{d\left(1+\frac{3k}{2\rho_2}\right)^2}{2t}.
\end{equation}

As a consequence, we can get the Liouville property if $\rho_1\ge0$, i.e. there is no positive  constant solution to $Lu=0$.
\end{remark}

Taking $a(t)=e^{\gamma
\rho_1t}\left(1-e^{\gamma \rho_1t}\right)^{\beta}$ with  $\beta>1$ even,  $\gamma\ge\frac{2\rho_2}{(\rho_2+k)(1+\beta)}$, then $\frac{a'}{a}=\frac{\gamma\rho_1((1+\beta)e^{\gamma\rho_1t}-1)}{e^{\gamma\rho_1t}-1}$, $\frac{\int_0^ta(s)ds}{a(t)}=\frac{1}{(1+\beta)\gamma\rho_1}(e^{\gamma\rho_1t}-1)e^{-\gamma\rho_1t}$ and
$\eta(t)=\frac{d\rho_1}{4\rho_2}\frac{(\gamma-2)\rho_2-\left(\gamma(\rho_2+k)(1+\beta)
-2\rho_2\right)e^{\gamma \rho_1 t}}{1-e^{\gamma\rho_1 t}}$. Hence Proposition \ref{thm-LY} reduces to

\begin{corollary}\label{cor2}
 Let $u$ be a positive solution to the heat
equation $\partial_tu=Lu$ such that $u_0:=u(\cdot,0)\in C^{\infty}(M)$ with $u_0\ge0$,
we have for $\beta>1 \mbox{ even},\gamma\ge\frac{2\rho_2}{(\rho_2+k)(1+\beta)}$,
\begin{align}\label{gradient2}
\Gamma(\log u)&-\frac{2\rho_2}{(1+\beta)\gamma \rho_1}\left(1-e^{\gamma\rho_1t}\right)e^{-\gamma\rho_1t}\Gamma^Z(\log u)
\le \alpha(t)(\log
u)_t+\varphi(t),\end{align}
where
$$
\alpha(t)=1+\frac{(1+\beta)k}{\beta\rho_2}e^{-\gamma\rho_1t}+\left(\frac{2}{(1+\beta)\gamma}-
\frac{k}{\rho_2}\right)(1-e^{\gamma\rho_1t})e^{-\gamma\rho_1t},
$$
and
$$\aligned
\varphi(t)&=\frac{d\rho_1\left(\gamma(\rho_2+k)(1+\beta)-2\rho_2\right)^2}{8\gamma(1+\beta)\rho_2^2}\left(1-e^{-\gamma\rho_1t}\right)
+\frac{d\rho_1\gamma(\rho_2\beta+k+k\beta)^2}{8\rho_2^2(\beta-1)}\frac{e^{-\gamma\rho_1t}}{e^{\gamma\rho_1t}-1}
\\&\ \ \ +\frac{d\rho_1\gamma(\gamma(\rho_2+k)(1+\beta)-2\rho_2)(\rho_2\beta+k+k\beta)}{4\beta\rho_2^2}e^{-\gamma\rho_1t}.
\endaligned$$
In particular, let $\gamma=\frac{2\rho_2}{(\rho_2+k)(1+\beta)}$ ($\beta>1$ even),  we have
\begin{align}\label{gradient3}
\Gamma(\log u)&+\frac{\rho_2+k}{ \rho_1}\left(1-e^{-\frac{2\rho_1\rho_2t}{(\rho_2+k)(1+\beta)}}\right)
\Gamma^Z(\log u)
\le \frac{k+\beta k+\beta\rho_2}{\beta\rho_2}e^{-\frac{2\rho_1\rho_2t}{(\rho_2+k)(1+\beta)}}(\log
u)_t\nonumber \\
&+\frac{d\rho_1(\beta\rho_2+k+k\beta)^2}{4\rho_2(\rho_2+k)(\beta+1)(\beta-1)}
\frac{e^{-\frac{2\rho_1\rho_2t}{(\rho_2+k)(1+\beta)}}}{e^{\frac{2\rho_1\rho_2t}{(\rho_2+k)(1+\beta)}}-1}.
\end{align}
It yields the following lower bound for $(\log u)_t$:
$$
(\log u)_t\ge- \frac{d(\beta\rho_2+k+k\beta)}{4(\rho_2+k)(\beta+1)(\beta-1)}\frac{\rho_1}{e^{\frac{2\rho_1\rho_2t}{(\rho_2+k)(1+\beta)}}-1}.
$$

\end{corollary}

Comparing with Corollary \ref{cor10}, \eqref{gradient2} works for both small time and large time $t$ while \eqref{gradient1} only works for finite time $t$ in the case of $\rho_1\ge0$.

%+++++++++++++++++++++++++++++++++++

%In particular, 增加BG文中的定理5.1及Parabolic
%Harnack不等式，如文BG中的定理6.1等等。

%+++++++++++++++++++++++++++++++++++++++++++++++++++

\section{Proof of Theorem \ref{thm-Schd}}\label{proof}

In this section, we give the proof of Theorem \ref{thm-Schd}. The method of the proof is inspired by \cite{Qian1,LX11}.

Assume $u$ is a positive solution to the Sch\"odinger equation $\partial_tu=L^Vu$.   Let $f=\log u$, we have
\begin{equation}\label{HJ2}
Lf  +\Gamma(f)=f_t+V,
\end{equation}
where $\Gamma$ is induced by $L$ rather than $L^V$. For some
positive function $b(t)$ and some functions $\alpha(t),\varphi(t)$,
define the parameter function:
$$ F=\Gamma(f)+b(t)\Gamma^Z(f)-\alpha(t)(f_t+V)-\varphi(t),
$$

it follows that
\begin{align*}
(L-\partial_t)F&= L\Gamma(f)-2\Gamma(f,f_t)+b(t)\left(L\Gamma^Z(f)-2\Gamma^Z(f,f_t)\right)-\alpha(t)(Lf_t-f_{tt})\\
&\hskip 12pt -\alpha(t)(LV-V_t)-b'(t)\Gamma^Z(f)+\alpha'(t)(f_t+V)+\varphi'(t).
\end{align*}

Applying \eqref{HJ2} and  assumption (H.2),  we have
\begin{align*}
(L-\partial_t)F&=2\Gamma_2(f)+2b(t)\Gamma_2^Z(f)-2\Gamma(f,F)-2(\alpha(t)-1)\Gamma(f,V)+2b(t)\Gamma^Z(f,V)\\
&\hskip 12pt -\alpha(t)LV-b'(t)\Gamma^Z(f)+\alpha'(t)f_t+\alpha'(t)V+\varphi'(t)
\end{align*}

Thanks to the  generalized curvature dimension inequality $CD(\rho_1,\rho_2,k,d)$,
\begin{align*}
(L-\partial_t) F&\ge-2\Gamma(f,F)+\frac{2}{d} (Lf)^2+\left(2\rho_1-\frac{2k}{b(t)}\right)\Gamma(f)+\left(2\rho_2-b'(t)\right)\Gamma^Z(f)\\
&\hskip 12pt -2(\alpha(t)-1)\Gamma(f,V)+2b(t)\Gamma^Z(f,V)-\alpha(t)LV+\alpha'(t) (V+f_t)+\varphi'(t)\\
&\ge -2\Gamma(f,F)-\frac{4\eta(t)}{d}Lf-\frac{2\eta^2(t)}{d}+\left(2\rho_1-\frac{2k}{b(t)}\right)\Gamma(f)+\left(2\rho_2-b'(t)\right)\Gamma^Z(f)\\
&\hskip 12pt -2(\alpha(t)-1)\Gamma(f,V)+2b(t)\Gamma^Z(f,V)-\alpha(t)LV+\alpha'(t) (V+f_t)+\varphi'(t),
\end{align*}
where $\eta=\eta(t)$ is some function to be determined later.  It follows by \eqref{HJ2},
\begin{align*}
(L-\partial_t) F&\ge -2\Gamma(f,F)+\left(\frac{4\eta}{d}+2\rho_1-\frac{2k}{b(t)}\right)\Gamma(f)
+(2\rho_2-b'(t))\Gamma^Z(f)-\left(\frac{4\eta}{d}-\alpha'(t)\right)(f_t+V)\nonumber\\
&\hskip 12pt -2(\alpha(t)-1)\Gamma(f,V)+2b(t)\Gamma^Z(f,V)-\alpha(t)LV-\frac{2\eta^2(t)}{d}+\varphi'(t).%\label{diff1}
\end{align*}

Applying \eqref{schd1} and the basic fact $2ax\le \frac{a^2}{\varepsilon}+\varepsilon x^2,\forall a,\varepsilon\in\mathbb{R}^+$, we have for any positive functions $\epsilon_1(t),\epsilon_2(t)$,
\begin{align*}
(L-\partial_t) F+2\Gamma(f,F)&\ge \left(\frac{4\eta}{d}+2\rho_1-\frac{2k}{b(t)}\right)\Gamma(f)
+\left(2\rho_2-b'(t)\right)\Gamma^Z(f)-\left(\frac{4\eta}{d}-\alpha'(t)\right)(f_t+V)\\
&\hskip12pt -2|\alpha(t)-1|\gamma_1(\Gamma(f))^{1/2}-2b(t)\gamma_2(\Gamma^Z(f))^{1/2}
-\alpha(t)\theta-\frac{2\eta^2(t)}{d}+\varphi'(t)\\
&\ge \left(\frac{4\eta}{d}+2\rho_1-\frac{2k}{b(t)}-\epsilon_1(t)\right)\Gamma(f)+
\left(2\rho_2-b'(t)-\epsilon_2(t)\right)\Gamma^Z(f)\\
&\hskip 12pt -\left(\frac{4\eta}{d}-\alpha'(t)\right)(f_t+V)-
\frac{|\alpha(t)-1|^2\gamma_1^2}{\epsilon_1(t)}-\frac{b^2(t)\gamma_2^2}{\epsilon_2(t)}
-\alpha(t)\theta-\frac{2\eta^2(t)}{d}+\varphi'(t).
\end{align*}

 For any $\epsilon_1,\epsilon_2\in (0,1)$,  Take
 \begin{eqnarray}
  \epsilon_1(t)&=&\frac{\epsilon_1 a'}{a}, \nonumber\\
  \epsilon_2(t)&=&2\epsilon_2\rho_2,\nonumber\\
  b(t)&=&\frac{2(1-\epsilon_2)\rho_2\int_0^ta(s)ds}{a(t)},\nonumber\\
    \eta(t)&=&\frac{d}{4}\left((1+\epsilon_1)\frac{a'}{a}+\frac{2k}{b(t)}-2\rho_1\right),\label{choice}\\
 \alpha(t)&=&\frac{4}{d\cdot a(t)}\int_0^ta(s)\eta(s)ds,\nonumber\\
  \varphi(t)&=& \frac{1}{a(t)}\int_0^ta(s)\left(\frac{\gamma_1^2a(s)|\alpha-1|^2}{\epsilon_1 a'(s)}+\frac{b^2\gamma_2^2}{2\epsilon_2\rho_2}+\theta\alpha(s)+\frac{2\eta^2(s)}{d}\right)ds.\nonumber
 \end{eqnarray}

 such that
  \begin{eqnarray*}
   \frac{4\eta}{d}+2\rho_1-\frac{2k}{b(t)}&=&(1+\epsilon_1)\frac{a'}{a},\nonumber\\
  2(1-\epsilon_2)\rho_2-b'(t)&=&\frac{a'b}{a},\nonumber\\
  \frac{4\eta}{d}-\alpha'(t)&=&\frac{a'\alpha}{a},\nonumber\\
  \varphi'(t)-\frac{|\alpha-1|^2\gamma_1^2}{\epsilon_1}-\frac{b^2\gamma_2^2}{\epsilon_2}-\alpha\theta-\frac{2\eta^2}{d}&=&-\varphi\frac{a'}{a}.\nonumber
  \end{eqnarray*}

With this choice, we have

\begin{lemma}
For the parameter function $F$, we have
\begin{align*}
(L-\partial_t) F+2\Gamma(f,F)&\ge\frac{a'}{a}F.
\end{align*}
\end{lemma}

Now we are ready to prove Theorem \ref{thm-Schd}.
\begin{proof}[Proof of Theorem \ref{thm-Schd}]
Note that  \begin{align*}
(L-\partial_t)(a(t)uF)&=a(t)u\big((L-\partial_t)F+2\Gamma(f,F)\big)-a'uF(u)+a(t)F(L-\partial_t)u\\
&\ge a(t)u F V,
\end{align*}
it implies $$
(L^V-\partial_t)(a(t)uF)\ge 0.
$$
Notice that $a(0)=0$, the desired result follows by the maximum
principle, see \cite{Friedman,Bony}. \end{proof}

\section{Harnack inequalities}\label{sec-harnack}
In this section, we shall derive the Harnack inequalities for the positive solutions to the Schr\"odinger equaition \eqref{schd0}, which are the consequence of the Li-Yau type inequality \eqref{LY1}. Harnack inequality  is one of the main techniques in the regularity theory of PDEs over  the past several decades.

To this end, we introduce the metric associated to Schr\"odinger operator $L^V$,  see \cite{LY} for elliptic operators on Riemannian manifolds. For $\delta >1$, denote
$$
\rho_\delta(x,y,t)=\inf_{\gamma\in\Gamma_0}\left\{\frac{\delta}{4t}\int_0^1\sum_{i=1}^ma_i^2(s)ds
+t\int_0^1V(\gamma(s),(1-s)t_2+st_1)ds\right\}
$$
where $\Gamma_0$ is a set of admissible curves $\gamma: [0,1]\to M$  satisfying $\gamma(0)=x,
\gamma(1)=y$, and $\gamma'(s)=\sum_{i=1}^ma_i(s)X_i(\gamma(s))$ a. e. s. Since $\sum_ia_i^2(s)$ does not depend on the choice of the $X_j$'s (see section 2 in \cite{Jerison}), nor does the distance function $\rho_\delta$. In the case of $V\equiv0$, $\rho_\delta$ is nothing less than the Carnot-Carath\'eodory distance $d_{CC}$ induced by the subelliptic operator $L$.

\begin{theorem}\label{harnack1} Let $u$ be a positive solution to the
Schr\"odinger equation \eqref{schd0}. There exit positive constants $C_i',i=1,\cdots,6$ and $\delta_0=\delta_0(\rho_2,k)>1$,  we have for all $0<t_1<t_2$,  $x,y\in M$ and $\delta>\delta_0$,
$$
u(x,t_1)\le u(y,
t_2)\left(\frac{t_2}{t_1}\right)^{\frac{C_1'}{\delta}}\exp\left(\rho_\delta(x,y,t_2-t_1)+\sum_{i=1}^5\frac{C_{i+1}'}{i}\left(t_2^{i}-t_1^{i}\right)\right).
$$

\end{theorem}

\begin{proof} Let $\gamma$ be any admissible curve given by $\gamma: \
[0,1]\to M$, with $\gamma(0)=y$ and $\gamma(1)=x$. We define $\eta:\
[0,1]\to M\times [t_1,t_2]$ by $\eta(s)=(\gamma(s),(1-s)t_2+st_1)$,
clear that $\eta(0)=(y, t_2)$ and $\eta(1)=(x,t_1)$. Integrating
$\frac{d}{ds}\log u$ along $\eta $, we get
$$
\aligned \log u(x,t_1)-\log u(u, t_2)&=\int_0^1\frac{d}{ds}\log
uds\\
&=\int_0^1 \sum_{i=1}^ma_iX_i\log u(x,s)-(t_2-t_1)(\log
u)_s ds
\endaligned
$$

Applying \eqref{cor0-dh} in Corollary \ref{cor0}, there exist some constants $C_i',i=1,\cdots,6$ and $\delta_0=\delta_0(\rho_2,k)>1$ such that for $\delta>\delta_0$, $t=(1-s)t_2+st_1$, we have
$$
\aligned \log \frac{u(x,t_1)}{u(x,t_2)} &\le \int_0^1 \sum_{i=1}^m|a_i||X_i\log u|
+(t_2-t_1)\left(-\frac{1}{\delta}\Gamma(\log u)+ \sum_{i=1}^6C_i't^{i-2}+V(\gamma(s),t) \right) ds\\
%&\hskip 12pt +(t_2-t_1)\int_0^1 V(\gamma(s),(1-s)t_2+st_1)ds \\
&\le \int_0^1   \sum_{i=1}^m\left(|a_i(s)||X_i\log
u|-\frac{t_2-t_1}{\delta}|X_i\log u|^2\right)+(t_2-t_1)V(\gamma(s),t) ds \\
&\hskip
12pt+C_2'(t_2-t_1)+C_1'\log \frac{t_2}{t_1}+\sum_{i=1}^4\frac{C_{i+2}'}{i+1}\left(t_2^{i+1}-t_1^{i+1}\right)\\
&\le\int_0^1 \frac{\delta}{4(t_2-t_1)}\sum_{i=1}^ma_i^2(s)+(t_2-t_1)V(\gamma(s),(1-s)t_2+st_1) ds
\\
&\hskip
12pt+C_1'\log \frac{t_2}{t_1}+\sum_{i=1}^5\frac{C_{i+1}'}{i}\left(t_2^{i}-t_1^{i}\right),
\endaligned
$$
taking exponentials of the above inequality gives the desired
result.
\end{proof}

Applying Corollary \ref{cor11} in the above proof, we have
\begin{theorem}\label{harnack2} Let $u$ be a positive solution to the Schr\"odinger equation $L^Vu=0$.  There exits positive constant $C=C(\theta,\gamma_1,\gamma_2,\rho_1,\rho_2,k,d)$, for any
 $x,y\in M$, we have
$$
u(x)\le u(y)\exp{(Cd(x,y))},
$$
where $d(x,y)=\inf_{\gamma\in\Gamma_0}\left\{\int_0^1\sum_{i=1}^ma_i^2(s)ds
+\int_0^1V(\gamma(s))ds\right\}$.

\end{theorem}

\section{Perelman type Entropy and its monotonicity}\label{entropy}

Throughout this section, we assume that the operator $L$ satisfies the generalized curvature-dimension inequality $CD(\rho_1,\rho_2, k,d)$ for some $\rho_1\ge0$, $\rho_2>0$, $k\ge0$ and $d<\infty$. Consequently, the invariant measure is finite, see \cite{BB2} Theorem 1.3.  Without loss any generality, we suppose $\mu$ is a probability measure.
Let $u(x,t)$ be the positive solution of the heat equation $\partial_t u=L u$ on $M\times [0,\infty)$, we can assume $\int_Mud\mu=1$. Denote $g(x,t)$ the function
\begin{equation}\label{s1}
u(x,t)=\frac{e^{-g(x,t)}}{(4\pi t)^{\frac{\tau D}{2}}}
\end{equation}
%利用（2.7）\eqref{gradient1},有可能将D的表达式表示为$D=\frac{d\gamma^2}{4(\gamma-1)}\left(1+\frac{(1+\gamma)k}{\gamma\rho_2}\right)^2$
%从而可以将\varsigma的范围扩展到$\rho_2\le \varsigma<2\rho_2$,当然其他系数将会发生变化。
with $D=d\left(1+\frac{3k}{2\rho_2}\right)^2$, where $\tau$ is some positive constant determined later. We define the so-called  Nash-type entropy, see \cite{Na,Ni},
\begin{equation*}\label{nash1}
N(u,t)=-\int_M u\log ud\mu
\end{equation*}
and
\begin{equation}\label{nash2}
\widetilde{N}(u,t)=N(u,t)-\frac{\tau D}{2}(\log (4\pi t)+1).
\end{equation}
By applying the Li-Yau type inequality established in section 2, we have the monotonicity of the Nash entropy formulae:
\begin{theorem}\label{mono}
For all $t>0$ and $\tau\ge1$, we have
\begin{equation*}\label{eq-mono}
\frac{d}{dt}\widetilde{N}(u,t)=\int_Mu\left(\Gamma(g)+\left(1+\frac{3k}{2\rho_2}\right)
g_t+\frac{3k\tau D}{4\rho_2t}\right)d\mu\le 0.
\end{equation*}
\end{theorem}
\begin{proof}

Denote $f=\log u$ as  above, we have $f=-g-\frac{\tau D}{2}\log (4\pi t)$, it yields
\begin{equation}\label{id1}
f_t=-g_t-\frac{\tau D}{2t}, \ \Gamma(f)=\Gamma(g).
\end{equation}

Thanks to the fact $\int_Mud\mu=1$, the integration by parts formula gives
$$
\frac{d}{dt}\widetilde{N}(u,t)=\int_M u\Gamma(\log u)d\mu-\frac{\tau D}{2t}=\int_M u\Gamma(g)d\mu-\frac{\tau D}{2t},
$$
\begin{equation*}\label{id3}
\int_M ug_t d\mu=-\frac{\tau D}{2t}-\int_M uf_t d\mu=-\frac{\tau D}{2t}-\int_M u_td\mu=-\frac{\tau D}{2t}.
\end{equation*}

It follows
\begin{align*}%\label{id4}
\frac{d}{dt}\widetilde{N}(u,t)&=\int_Mu\left(\Gamma(g)\right)d\mu+\left(1+\frac{3k}{2\rho_2}\right)
\left(\int_Mu g_td\mu+\frac{\tau D}{t}\right)-\frac{\tau D}{2t}\nonumber\\
%&=\int_Mu\left(\Gamma(g)+\left(1+\frac{3k}{2\rho_2}\right)g_t+\left(1+\frac{3k}{2\rho_2}\right)\frac{\tau d}{t}-\frac{\tau d}{t}\left(1+\frac{3k}{2\rho_2}\right)^2\right)d\mu\\
&=\int_Mu\left(\Gamma(g)+\left(1+\frac{3k}{2\rho_2}\right)g_t+\frac{3k\tau D}{4\rho_2t}\right)d\mu
\end{align*}
By \eqref{gradient1-100} and \eqref{id1}, we have

\begin{equation*}\label{gradient1-1}
\Gamma(g)+\left(1+\frac{ 3k}{2\rho_2}\right)g_t+\left(1+\frac{ 3k}{2\rho_2}\right)\frac{\tau D}{2t}-\frac{D}{2t}\le0.
\end{equation*}
For $\tau\ge1$, clearly we have
$$
\frac{d}{dt}\widetilde{N}(u,t)\le 0.
$$
 The desired conclusion follows.\end{proof}

Following Perelman \cite{Pl}, we define
\begin{equation}\label{perel1}
\mathcal{W}(u,t)=\int_Mu\left(t\Gamma(g)+g-\tau D\right)d\mu=\frac{d}{dt}\left(t\widetilde{N}(u,t)\right),
\end{equation}
and
\begin{equation}\label{perel2}
\mathcal{W}_{\varsigma}(u,t)=\mathcal{W}(u,t)+\varsigma t^2\int_Mu\Gamma^Z(\log u)d\mu,
\end{equation}
where $\varsigma$ is some positive constant to be determined.

To study the monotonicity of the functional $\mathcal{W}_{\varsigma}$, let us give the following useful lemma.
\begin{lemma}\label{g20} For any positive solution to the heat equation $\partial_tu=Lu$, we have
\begin{equation*}
\frac{d^2}{dt^2}N(u,t)=-2\int_Mu\Gamma_2(\log u)d\mu,
\end{equation*}
this gives
\begin{equation*}\label{g21}
\frac{d}{dt}\mathcal{W}(u,t)=2\int_M u\Gamma(\log u)d\mu-2t\int_Mu\Gamma_2(\log u)d\mu-\frac{\tau D}{2t}.
\end{equation*}
Moreover, denote $\mathcal{B}(u,t)=\int_Mu\Gamma^Z(\log u)d\mu$, we have
\begin{equation*}\label{g22}
\mathcal{B}'(u,t)=-2\int_Mu\Gamma_2^Z(\log u)d\mu
\end{equation*}
\end{lemma}
\begin{proof}Compute
$$
\aligned
\frac{d^2}{dt^2} {N}(u,t)&=\frac{d^2}{dt^2} N(u,t)\\
&=-\frac{d}{dt}\int_M uL\log ud\mu \\
&=-\int_M LuL\log ud\mu-\int_Mu\partial_t(L\log u)d\mu \\
&= -2\int_M LuL\log ud\mu-\int_Mu(\partial_t-L)(L\log u)d\mu .
\endaligned
$$
Notice that
$$
(\partial_t-L)L\log u=L(\partial_t-L)\log u=L\Gamma(\log u),
$$
this yields
\begin{align}
\frac{d^2}{dt^2} {N}(u,t)&=-\int_MuL\Gamma(\log u)d\mu+2\int_M u\Gamma(\log u,L\log u)d\mu \nonumber\\
&=-2\int_Mu\Gamma_2(\log u)d\mu.\nonumber
\end{align}
Hence
\begin{align*}
\frac{d}{dt}\mathcal{W}(u,t)&=\frac{d^2}{dt^2}\left(t\widetilde{N}(u,t)\right)\\
&=tN''(u,t)+2N'(u,t)-(t\tau D(1+\log(4\pi t))''\\
&=-2t\int_Mu\Gamma_2(\log u)d\mu+2\int_Mu\Gamma(\log u)d\mu-\frac{\tau D}{2t}.
\end{align*}
Meanwhile
$$\aligned
B'(u,t)&=\int_M\partial_tu\Gamma^Z(\log u)d\mu+2\int_Mu\Gamma(\log u,\partial_t\log u)d\mu\\
%&=\int_M Lu\Gamma^Z(\log u)d\mu+2\int_Mu\Gamma(\log u,\partial_t\log u)d\mu\\
&=\int_MLu \Gamma^Z(\log u)d\mu+2\int_Mu\Gamma^Z\Big(\log u,L\log u+\Gamma(\log u)\Big)d\mu\\
&=\int_MuL\Gamma^Z(\log u)d\mu+2\int_Mu\Gamma^Z\Big(\log u,L\log u\Big)d\mu+2\int_Mu\Gamma^Z\Big(\log u,\Gamma(\log u)\Big)d\mu\\
&=\int_MuL\Gamma^Z(\log u)d\mu+2\int_Mu\Gamma^Z\Big(\log u,L\log u\Big)d\mu+2\int_Mu\Gamma\Big(\log u,\Gamma^Z(\log u)\Big)d\mu,
\endaligned$$
where in the last equality we apply the assumption (H. 2).
Notice that $$
\int_MuL\Gamma^Z(\log u)d\mu=-\int_M\Gamma(u,\Gamma^Z(\log u))d\mu=-\int_Mu\Gamma(\log u,\Gamma^Z(\log u))d\mu,
$$
it follows
\begin{align}
 B'(u,t)&=-\int_MuL\Gamma^Z(\log u)d\mu+2\int_Mu\Gamma^Z\Big(\log u,L\log u\Big)d\mu\nonumber\\
&=-2\int_Mu\Gamma_2^Z(\log u)d\mu.\nonumber
\end{align}
Hence we complete the proof.\end{proof}

Now we are ready to state the main result in this section.
\begin{theorem}\label{thm-perel}
For $\rho_2\le \varsigma\le \frac{5}{3}\rho_2$ and $\tau\ge 2+\frac{2k}{\rho_2}$, we have
\begin{align*}
\frac{d}{dt}\mathcal{W}_{\varsigma}(u,t)\le -\frac{2t}{d}\int_Mu(L\log u)^2d\mu+\frac{D}{2t}\left(\frac{2(k+\varsigma)}{\varsigma}-\tau\right)\le 0
\end{align*}
holds for all $t>0$.
\end{theorem}
\begin{proof}
By Lemma \ref{g20}, we have
\begin{align*}
\frac{d}{dt}\mathcal{W}_{\varsigma}(u,t)&=\frac{d}{dt}\mathcal{W}(u,t)+2\varsigma tB(t)+\varsigma t^2\frac{d}{dt}\mathcal{B}(t)\nonumber\\
&=-2t\int_Mu\Gamma_2(\log u)d\mu-2\varsigma t^2\int_Mu\Gamma_2^Z(\log u)d\mu+2\int_M u\Gamma(\log u)d\mu\nonumber\\
&\hskip 12pt +2\varsigma t\int_Mu\Gamma(\log u)d\mu-\frac{\tau D}{2t}\\
&\le -\frac{2t}{d}\int_Mu(L\log u)^2d\mu+\frac{2(k+\varsigma)}{\varsigma}\int_Mu\Gamma(\log u)d\mu
+2t(\varsigma-\rho_2)\int_Mu\Gamma^Z(\log u)d\mu-\frac{\tau D}{2t}.\nonumber
%&=-\beta''\int_Mu\log ud\mu+2\beta'\int_Mu\Gamma(\log u)d\mu+\alpha'\int_Mu\Gamma(\log u)d\mu-(\beta\gamma)''\nonumber\\
%&\hskip 12pt -2\beta\int_Mu\Gamma_2(\log u)d\mu-2\alpha\int_Mu\Gamma_2^Z(\log u)d\mu
\end{align*}
The last inequality follows from the generalized curvature-dimension inequality $CD(\rho_1,\rho_2,k,d)$ with $\rho_1\ge0$.
Applying the Li-Yau type estimate \eqref{gradient1-100}, we have for $\rho_2\le \varsigma\le \frac{5}{3}\rho_2$,
\begin{align*}
\frac{2(k+\varsigma)}{\varsigma}\int_Mu\Gamma(\log u)d\mu
+2t(\varsigma-\rho_2)\int_Mu\Gamma^Z(\log u)d\mu&\le \frac{(k+\varsigma)(2\rho_2+3k)}{\varsigma\rho_2}\int_M u_td\mu+\frac{(k+\varsigma)D}{\varsigma t}\\
&=\frac{(k+\varsigma)D}{\varsigma t}.
\end{align*}
This gives
\begin{align*}
\frac{d}{dt}\mathcal{W}_{\varsigma}(u,t)\le -\frac{2t}{d}\int_Mu(L\log u)^2d\mu+\frac{D}{2t}\left(\frac{2(k+\varsigma)}{\varsigma}-\tau\right),
\end{align*}
hence for $\tau\ge 2+\frac{2k}{\rho_2}\ge 2+\frac{2k}{\varsigma}$, we have
$$
\frac{d}{dt}\mathcal{W}_{\varsigma}(u,t)\le0.
$$
We complete the proof.
\end{proof}
\begin{remark}

(i).  Theorem \ref{mono} and Theorem \ref{thm-perel} generalize Theorem 1.12 and Theorem 1.13 in \cite{CKL} respectively, where they consider the case of  sub-Laplacians on a closed pseudohermitian 3-manifold with Tanaka-Webster curvature bounded below, i.e. $CD(k,\frac12,1,2)$.

(ii). For the diffusion operators satisfying the curvature-dimension inequality  $CD(k,m)$ on a complete noncompact Riemannian manifold, the monotonicity of the Perelman entropy has been obtained in \cite{L1}, see also \cite{BG2,L2}. It would be very interesting to derive monotonicity of the Perelman type entropies  for subelliptic operator on complete but noncompact Riemannian manifolds, which will be studied in near future.

(iii). In the definitions of  \eqref{s1}, \eqref{nash2} and \eqref{perel1}, if we replace the parameter $D=d\left(1+\frac{3k}{2\rho_2}\right)^2$ by $\widetilde{D}=\frac{d\gamma^2}{4(\gamma-1)}\left(1+\frac{(1+\gamma)k}{\gamma\rho_2}\right)^2$ ($\gamma>1$), using the Li-Yau type inequality \eqref{gradient1-10} instead of \eqref{gradient1-100} in the above proofs,  we have for all $t>0$, $\gamma>1$ and $\tau\ge1$,
\begin{equation*}
\frac{d}{dt}\widetilde{N}(u,t)=\int_M u\left(\Gamma(g)+\left(1+\frac{(1+\gamma)k}{\gamma\rho_2}\right)g_t+\frac{(\gamma+1)k\tau \widetilde{D}}{2\gamma\rho_2t}\right)d\mu\le0.
\end{equation*}

Moreover, for $\gamma>1$, $\rho_2\le \varsigma\le \frac{3+\gamma}{1+\gamma}\rho_2$ and $\tau\ge 2+\frac{2k}{\rho_2}$, we have
 \begin{align*}
\frac{d}{dt}\mathcal{W}_{\varsigma}(u,t)\le -\frac{2t}{d}\int_Mu(L\log u)^2d\mu+\frac{\widetilde{D}}{2t}\left(\frac{2(k+\varsigma)}{\varsigma}-\tau\right)\le 0.
\end{align*}

If we choose $\gamma=2$, the above two results become Theorem \ref{mono} and Theorem \ref{thm-perel} respectively.
\end{remark}

%{\bf Acknowledgement:}  The author would like to express his sincere thanks to Prof. F. Baudoin for the helpful discussion, \cite{B1}.
%The author is greatly indebted to Prof. L. M. Wu (Clermont-Ferrand University II, France), Prof. D.
%Bakry (Toulouse University III, France) and Prof. H. Q. Li
%(Fudan University, China) for their constant encouragement and support.


\begin{thebibliography}{}
\bibitem{Pl} G. Perelman, The entropy formula for the Ricci flow and its geometric applications, arxiv:0211159.
\end{thebibliography}
\end{document}